\documentclass[psamsfonts,a4paper,11pt]{amsart}
\usepackage[all,arc]{xy}
\usepackage[all]{xy}
\usepackage{latexsym}
\usepackage{amssymb}
\usepackage{amsmath}
\newtheorem{definition}{Definition}[section]
\newtheorem{lemma}[definition]{Lemma}
\newtheorem{prop}[definition]{Proposition}
\newtheorem{theorem}[definition]{Theorem}
\newtheorem{cor}[definition]{Corollary}
\newtheorem{conj}[definition]{Conjecture}

\newtheorem{ques}[definition]{Question}

\theoremstyle{definition}
\newtheorem{fact}[definition]{Fact}

\newcommand*{\lcm}{\mathop{{\rm lcm}}\nolimits}
\newcommand*{\Jac}{\mathop{{\rm Jac}}\nolimits}
\newcommand*{\rad}{\mathop{{\rm rad}}\nolimits}

\newcommand{\Zg}{\mathop{{\rm Zg}}\nolimits}

\newcommand*{\ex}{\exists}
\newcommand*{\fa}{\forall}

\newcommand*{\ov}{\overline}
\newcommand*{\seq}{\subseteq}

\newcommand*{\sq}{\sqrt}

\newcommand*{\fty}{\infty}

\newcommand*{\wg}{\wedge}
\newcommand*{\wh}{\widehat}
\newcommand*{\wt}{\widetilde}

\newcommand*{\bsm}{\left(\begin{smallmatrix}}
\newcommand*{\esm}{\end{smallmatrix}\right)}
\newcommand*{\bp}{\begin{pmatrix}}
\newcommand*{\ep}{\end{pmatrix}}

\newcommand*{\mL}{\mathcal{L}}

\newcommand*{\A}{\mathbb{A}}

\newcommand*{\F}{\mathbb{F}}
\newcommand*{\N}{\mathbb{N}}
\newcommand*{\Q}{\mathbb{Q}}
\newcommand*{\R}{\mathbb{R}}
\newcommand*{\Z}{\mathbb{Z}}
\newcommand*{\Alg}{\mathbb{A}}

\newcommand*{\al}{\alpha}
\newcommand*{\be}{\beta}

\newcommand*{\ga}{\gamma}
\newcommand*{\Ga}{\Gamma}

\renewcommand*{\phi}{\varphi}

\renewcommand*{\th}{\theta}

\begin{document}

\footskip=30pt

\date{}

\title[]{On the decidability of the theory of modules over the ring of algebraic integers}

\author[]{Sonia L'Innocente, Carlo Toffalori}

\address[S.~L'Innocente, C.~Toffalori]{University of Camerino, School of Science and Technologies,
Division of Mathematics, Via Madonna delle Carceri 9, 62032 Camerino, Italy}

\email{sonialinnocente@unicam.it, carlo.toffalori@unicam.it}

\author[]{Gena Puninski}

\address[G.~Puninski]{Belarusian State University, Faculty of Mechanics and Mathematics, av. Nezalezhnosti 4,
Minsk 220030, Belarus}

\email{punins@mail.ru}

\thanks{The first two authors were supported by Italian
PRIN 2012 and GNSAGA-INdAM}



\begin{abstract}
We will prove that the theory of all modules over the ring of algebraic integers is decidable.
\end{abstract}

\maketitle

\section{Introduction}\label{S-intro}

Our original intention was to prove that the theory of all modules over the ring of algebraic integers is decidable.
When working out the proofs we found that they are applicable to a wider class of natural examples, hence the
resulting theorem sounds as follows. Suppose that $B$ is an effectively given (hence countable) B\'ezout domain such
that each nonzero prime ideal $P$ of $B$ is maximal, the residue field
$B/P$ is infinite, and the maximal ideal of
the localization $B_P$ is not finitely generated. Then the theory of all $B$-modules is decidable.

For instance this applies to the ring of algebraic integers $\A= \wt\Z$, and also to the algebraic closure of the polynomial ring $\F_p[t]$ over a prime field, i.e.\ its integral closure in the algebraic closure of its field of fractions $\F_p(t)$.

The reason why we require each nonzero prime ideal of $B$ to be maximal is the following. It has been noticed by
Gregory \cite{Gre15} (generalizing an early remark in \cite{PPT}) that the prime radical relation $a\in \rad(b)$
in a commutative ring $R$ is interpretable in the first order theory of its modules. If each nonzero prime ideal
of a B\'ezout domain $B$ is maximal then this relation is first order in the language of rings, and is decidable
in the examples of our interest.

However one should be aware that in general only few fragments of the theory of a ring are interpretable in the
language of its modules: for instance the theory of integers as a ring is undecidable, but the theory of abelian
groups is decidable by a classical result by Szmielew.

The proof of the main result uses basic algebraic facts and two pieces from the model theory of modules. The first
one is the model theory of modules over commutative B\'ezout domains recently developed in \cite{P-T15}. The second
is the decidability (or rather the proof of it in \cite{PPT} given in `elementary geometry' terms) of the theory of
modules over an effectively given valuation domain whose value group is archimedean and dense. In this paper we
deal with such domains locally: each localization $B_P$ is of this form. Thus the decidability proof is a kind of
local-global analysis which reduces our original question to a problem of covering intervals in value groups over
some constructive set of maximal ideals.

Working on this text we were much impressed by van den Dries--Macintyre \cite{D-M} and Prestel--Schmid \cite{P-S}
papers, from which few proofs and most examples are borrowed. The authors are also indebted to Angus Macintyre for
reading a preliminary version of this paper and drawing our attention to \cite{Rab}.

There is a clear need to relax the countability of the ring (hence the countability of the language) assumption
when talking about the decidability of the theory of its modules. It would allow to include into consideration 
other prominent examples of B\'ezout domains, such as rings of entire (real or complex) functions and the ring
of analytic real functions. However this requires to recast some foundations of model theory for modules starting
with its milestone, the Baur--Monk theorem, - so we postpone this discussion to a forthcoming paper.

\section{Valuation domains}\label{S-val}

All rings in this paper will be commutative with unity, and all modules will be unital with multiplication written 
on the right.

For basics in model theory of modules the reader is referred to Mike Prest's book \cite{Preb2}.  For instance
the \emph{Ziegler spectrum} of $R$, $\Zg_R$, is a topological space whose points are (isomorphism types of)
indecomposable pure injective modules and the topology is given by basic open (quasi-compact) sets $(\phi/\psi)$,
where $\phi$ and $\psi$ are pp-formulae in one variable. This open set consists of modules $M\in \Zg_R$ such that
$\phi(M)$ properly contains its intersection with $\psi(M)$.

A ring is said to be a \emph{domain}, if it has no zero divisors; and a \emph{valuation domain} if its ideals are
linearly ordered by inclusion. Some basics  of the theory of valuation domains can be found in \cite[Ch.~6]{F-S}.
For instance the Jacobson radical of $V$, $\Jac(V)$, is the largest proper ideal of $V$. Furthermore the
\emph{value group} $\Ga$ of $V$ is a linearly ordered (additive) abelian group such that each element $a\in V$ is assigned its
value $v(a)\in \Ga$, for instance $v(0)= \fty$ and $v(1)=0$. The positive cone $\Ga^+$ of $V$  can be
identified (as a poset) with the chain of proper principal ideals of $V$.

Recall that $\Ga$ is said to be \emph{archimedean} if for each $\al, \be\in \Ga^+$ there is a positive integer $n$
such that $n\al> \be$. In terms of $V$ this means that, for all nonzero noninvertible $a, b\in V$, there is an $n$
such that $a^n\in bV$, i.e.\ the \emph{prime radical relation} $a\in \rad(b)$ is trivial. Another way to say this
is that the only nonzero prime ideal of $V$ is the maximal one.

We say that $\Ga$ is \emph{dense} if for any $\al, \be\in \Ga^+$ there exists $\ga\in \Ga$ such that
$\al< \ga< \be$. In terms of $V$ this means that one could insert a principal ideal $cV$ strictly between
any two principal ideals $bV\subset aV$. Furthermore this holds true if and only if $\Jac(V)$ is not a
principal ideal.

It is known (see \cite[Sect.~II.2]{F-S}) that a linearly ordered abelian group $\Ga$ is archimedean if and only if
it is isomorphic to a subgroup of $(\R,+, \leq)$. Furthermore if $\Ga$ is additionally countable and dense, then
it is isomorphic to $(\Q,+, \leq)$.

Let $V$ be a valuation domain whose value group $\Ga$  is isomorphic to $(\Q,+, \leq)$. In this section we will
give a detailed description of $\Zg_V$ as a topological space. Some basic facts on model theory of modules over
valuation domains (including the description of indecomposable pure injective modules) can be found in
\cite[Sect.~11]{Punb} or \cite[Sect.~4]{PPT} - we will use them freely.

For instance the points of $\Zg_V$ are classified using pairs $(J,I)$ of ideals of $V$. Taking values we see that
each nonzero ideal of $V$ defines a cut (i.e.\ an ideal) on the positive cone $\Q^+$ of the value group of $V$.
In particular positive rationals correspond to nonzero principal ideals of $V$. In fact two types of ideals arise
in this way.

First every positive real $s$ defines a cut on $\Q^+$ (and hence an associated ideal) whose upper part consists of
the rationals $q$ (including $\fty$) such that $s \leq q$,  - we will denote this cut also by $s$. Next assume that
$q$ is a nonnegative rational and choose $b\in V$ whose value equals $q$. Then the ideal $b\Jac(V)$ defines a cut
`to the right of $q$' whose  upper part is the open interval $(q,\fty]$, - we will denote this ideal by $q_+$. For
instance the Jacobson radical of $V$ defines the cut $0_+$.

We will show a pair of cuts $(I,J)$ as a point on the quarter plane $\wh\Ga^+\times \wh\Ga^+$, where $\wh\Ga^+$
is obtained from $\Ga^+$ by completing by cuts.

$$
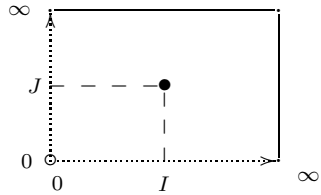
\begin{xy}
(0,0)*+={}="sy";(0,20)*+={.}="ty";%
(0,0)*+={\circ}="sx";(30,0)*+={.}="tx";(30,20)*+={.}="B";%
{\ar@{.>} "sx";"tx"};
{\ar@{.>} "sy";"ty"};
{\ar@{-} "ty";"B"};
{\ar@{-} "tx";"B"};
(-3,0)*+={_0};(1,-3)*+={_0};
(-4,20)*+={_{\infty}};(34,-2)*+={_\infty};
(15,0)*+={}="I";(0,10)*+={}="J";(15,10)*+={\bullet}="Z";%
{\ar@{--} "I";"Z"};{\ar@{--} "Z";"J"};
(15,-3)*+={_I};(-2,10)*+={_J};
\end{xy}
$$

\vspace{3mm}

For instance the pairs of positive reals give usual points, but the pair $(0_+, 0_+)$ represents the
infinitesimal point which is infinitely close to the origin.

The points of the Ziegler spectrum correspond to equivalence classes of pairs of cuts $(J,I)$.

Here the pair $(I,J)$ is equivalent to the pair $(K,L)$ if and only if there are $a \notin I$, $b \notin J$
such that (with respect to the multiplication in $V$) either 1) $I= Ka$, $Ja= L$ or 2) $Ib= K$, $J= Lb$. Thus these classes can be thought of as
lines $x+y= const$ of slope $-1$ on this quarter plane.

Here is the equivalence class of the point $(3,2)$.

$$
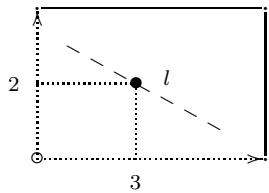
\begin{xy}
(0,0)*+={}="sy";(0,20)*+={.}="ty";%
(0,0)*+={\circ}="sx";(30,0)*+={.}="tx";(30,20)*+={.}="B";%
{\ar@{.>} "sx";"tx"};
{\ar@{.>} "sy";"ty"};
{\ar@{-} "ty";"B"};
{\ar@{-} "tx";"B"};
(4,15)*={}="bl";(24,4)*={}="el";(13,10)*+={\bullet};%
{\ar@{--}^(.6)l "bl";"el"};
(0,10)*+={}="a";(13,0)*+={}="b";(13,10)*+={}="z";%
{\ar@{.} "a";"z"};{\ar@{.} "b";"z"};
(-3,10)*+={_2};(13,-3)*+={_3};
\end{xy}
$$

\vspace{3mm}

Note that some lines degenerate. For instance the point $(0_+,0_+)$ comprises the whole line; and the same is
true for $(\fty, \fty)$, which we abbreviate as $\fty$. Furthermore all top points $(q,\fty)$, $q\in \Q^+$ comprise
one line; and the same holds true for points $(q_+, \fty)$. Similar considerations apply to points on the right
boundary $(\fty, q)$ and $(\fty,q_+)$.

The topology on $\Zg_V$ is given (with respect to the addition in $\Q$) by basic open sets (rectangles)
$(a, a+g]\times (b, b+h]$, where $a, b$ are nonnegative rationals and $g, h$ are positive rationals or $\fty$.
Here a point $(J,I)$ belongs to this open set if it can be slided within the rectangle along the line, i.e.\ if
there is an equivalent pair $(L,K)$ such that  $a\notin v(L)$, $a+g\in v(L)$ and $b\notin v(K)$, $b+h\in v(K)$.

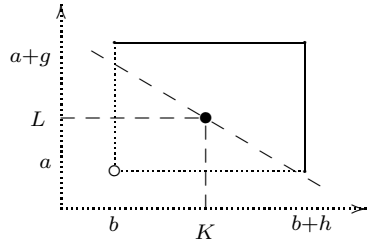
$$
\begin{xy}
(7,22)*={.}="A";(32,22)*={.}="B";%
(7,5)*+={\circ}="C";(32,5)*={.}="D";%
{\ar@{-} "A";"B"};{\ar@{-} "B";"D"};
{\ar@{.} "A";"C"};{\ar@{.} "C";"D"};
(0,0)*+={}="sy";(0,27)*+={}="ty";%
(0,0)*+={}="sx";(40,0)*+={}="tx";%
{\ar@{.>} "sx";"tx"};{\ar@{.>} "sy";"ty"};
(-4,20)*+={_{a+g}};(-2,6)*+={_a};%
(7,-2)*+={_b};(33,-2)*+={_{b+h}};%
(0,12)*+={}="L";(19,0)*+={}="K";(19,12)*+={\bullet}="Z";%
{\ar@{--} "L";"Z"};{\ar@{--} "K";"Z"};
(19,-3)*+={_K};%
(-3,12)*+={_L};%
(4,21)*+={}="bl";(34,3)*+={}="el";%
{\ar@{--} "bl";"el"};
\end{xy}
$$

\vspace{3mm}

For instance the point $(\sq 2, \sq 5)$ belongs to the open set $(2,3]\times (1,2]$ because $2< \sq 2+ 1 \leq 3$ and $1< \sq 5-1\leq 2$.

A word of caution: two sides of the rectangle are dashed, hence deleted, except of the two corner points. Thus on
the following diagram the line $l_1: x+y= 6$ intersects the rectangle $(2,4]\times (1,2]$, but the line
$l_2: x+y= 3$ does not.

$$
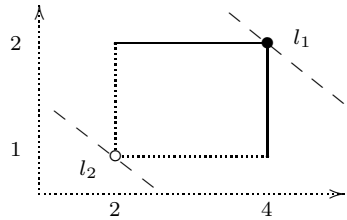
\begin{xy}
(10,20)*={.}="A";(30,20)*={\bullet}="B";%
(10,5)*+={\circ}="C";(30,5)*={.}="D";%
{\ar@{-} "A";"B"};{\ar@{-} "B";"D"};
{\ar@{.} "A";"C"};{\ar@{.} "C";"D"};
(25,24)*+={}="bl1";(40,12)*+={}="el1";%
(2,11)*+={}="bl2";(15,1)*+={}="el2";%
{\ar@{--}^{l_1} "bl1";"el1"};{\ar@{--}_{l_2} "bl2";"el2"};
(0,0)*+={}="sy";(0,25)*+={}="ty";%
(0,0)*+={}="sx";(40,0)*+={}="tx";%
{\ar@{.>} "sx";"tx"};
{\ar@{.>} "sy";"ty"};
(-3,20)*+={_{2}};(-3,6)*+={_1};%
(10,-2)*+={_2};(30,-2)*+={_{4}};%
\end{xy}
$$

\vspace{3mm}

Note the the infinite vertical strip $(1,2]\times (0,\fty]$ includes a finite point $(t,s)$ iff $t+s>1$.
Furthermore among infinite points it includes only top infinite points $(t,\fty), (t_+,\fty)$ and $\fty$.

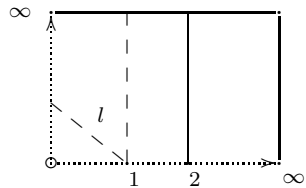
$$
\begin{xy}
(0,0)*+={}="sy";(0,20)*+={.}="ty";%
(0,0)*+={\circ}="sx";(30,0)*+={.}="tx";(30,20)*+={.}="B";%
{\ar@{.>} "sx";"tx"};
{\ar@{.>} "sy";"ty"};
{\ar@{-} "ty";"B"};
{\ar@{-} "tx";"B"};
(-4,20)*+={_{\infty}};(32,-2)*+={_\infty};
(10,0)*+={}="1";(18,0)*={.}="2";(10,20)*={}="t1";(18,20)*={.}="t2";%
{\ar@{--} "1";"t1"};{\ar@{-} "2";"t2"};
(0,8)*+={}="bl";(10,0)*+={.}="el";{\ar@{--}^l "bl";"el"};
(11,-2)*+={_1};(19,-2)*+={_2};
\end{xy}
$$

\vspace{3mm}

On the following diagram we demonstrate the best attempt to separate points $a=(1,1_+)$, $b= (1_+,1_+)$,
$c= (1,1)$ and $d= (1_+,1)$.

$$
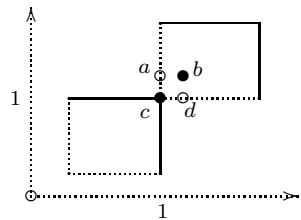
\begin{xy}
(0,0)*+={}="sy";(0,25)*={}="ty";%
(0,0)*+={\circ}="sx";(35,0)*={}="tx";%
{\ar@{.>} "sx";"tx"};
{\ar@{.>} "sy";"ty"};
(5,13)*={.}="A";(17,13)*={.}="B";(5,3)*={.}="C";(17,3)*={.}="D";%
{\ar@{-} "A";"B"};{\ar@{.} "C";"D"};{\ar@{.} "A";"C"};{\ar@{-} "B";"D"};
(17,23)*={.}="A1";(30,23)*={.}="B1";(17,13)*={.}="C1";(30,13)*={.}="D1";%
{\ar@{-} "A1";"B1"};{\ar@{.} "C1";"D1"};{\ar@{.} "A1";"C1"};{\ar@{-} "B1";"D1"};
(17,16)*={\circ};(15,17)*={_a};(20,16)*={\bullet};(22,17)*={_b};
(17,13)*={\bullet};(15,11)*={_c};(20,13)*={\circ};(21,11.5)*={_d};
(17.4,-2)*={_1};(-2,13)*={_1};
\end{xy}
$$

\vspace{3mm}

We see that the points $a, d$ are topologically indistinguishable and belong to the closures of both $b$ and $c$.

From this description it follows that $\Zg_V$ has a countable basis. We pinpoint another important property of
this basis. Following Hochster \cite{Hoch}, we call a topological space \emph{spectral} if it has a basis
consisting of compact open sets which is closed with respect to finite intersections (we drop the $T_0$-requirement
from the original definition).

\begin{lemma}\label{sober}
$\Zg_V$ is spectral.
\end{lemma}
\begin{proof}
A quick glance at the following diagram will convince the reader.

$$
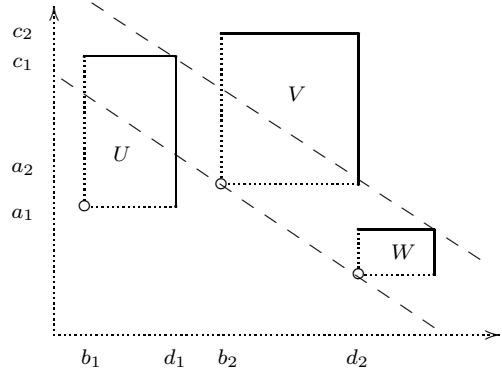
\begin{xy}
(0,0)*+={}="sx";(60,0)*+{}="tx";(0,45)*+{}="ty";%
{\ar@{.>} "sx";"tx"};
{\ar@{.>} "sx";"ty"};
(22,40)*={.}="A";(40,40)*={.}="B";%
(22,20)*+={\circ}="C";(40,20)*={.}="D";%
{\ar@{-} "A";"B"};{\ar@{-} "B";"D"};
{\ar@{.} "A";"C"};{\ar@{.} "C";"D"};
(4,37)*={.}="a";(16,37)*={.}="b";%
(4,17)*+={\circ}="c";(16,17)*={.}="d";%
{\ar@{-} "a";"b"};{\ar@{-} "b";"d"};
{\ar@{.} "a";"c"};{\ar@{.} "c";"d"};
(40,14)*={.}="e";(50,14)*={.}="f";%
(40,8)*+={\circ}="g";(50,8)*={.}="h";%
{\ar@{-} "e";"f"};{\ar@{-} "f";"h"};
{\ar@{.} "e";"g"};{\ar@{.} "g";"h"};
(1,34)*+={}="bl1";(50,1)*+={}="el1";%
{\ar@{--} "bl1";"el1"};
(5,44)*+={}="bl2";(56,10)*+={}="el2";%
{\ar@{--} "bl2";"el2"};
(9,24)*+={_U};(32,32)*+={_V};(46,11)*+={_W};
(-4,16)*+={_{a_1}};(-4,22)*+={_{a_2}};(-4,36)*+={_{c_1}};(-4,40)*+={_{c_2}};
(5,-3)*+={_{b_1}};(16,-3)*+={_{d_1}};
(23,-3)*+={_{b_2}};(40,-3)*+={_{d_2}};
\end{xy}
$$

\vspace{3mm}

\noindent
Here the basic open set corresponding to the rectangle $W$ is the intersection of basic open sets for rectangles
$U$ and $V$.
\end{proof}

We can `flatten' the above diagrams by projecting all objects onto the diagonal $x=y$. Namely each line
$x+y= const$ will go to this constant. Then the rectangle $(a, a+g]\times (b, b+h]$ will get projected to
the half-open interval $(a+b, a+b+g+h]$.

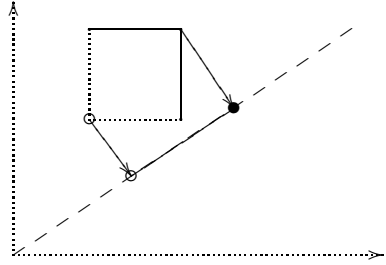
$$
\begin{xy}
(0,0)*+={}="sy";(0,35)*+{}="ty";%
(0,0)*+={}="sx";(50,0)*+{}="tx";(46,31)*+{}="d";%
{\ar@{.>} "sx";"tx"};
{\ar@{.>} "sy";"ty"};
{\ar@{--} "sx";"d"};
(10,30)*={.}="A";(22,30)*={.}="B";(10,18)*={\circ}="C";(22,18)*={.}="D";%
{\ar@{-} "A";"B"};{\ar@{.} "C";"D"};{\ar@{.} "A";"C"};{\ar@{-} "B";"D"};
(15.5,10.5)*={\circ}="E";(29,19.5)*={\bullet}="F";
{\ar@{->} "C";"E"};{\ar@{->} "B";"F"};{\ar@{-} "E";"F"}
\end{xy}
$$

\vspace{3mm}

For instance on the previous diagram we obtain $U= (a_1+ b_1, c_1+ d_1]$ and $V= (a_2+ b_2, c_2+ d_2]$. Thus
$W$ corresponds to the interval $(a_2+ b_2, c_1+ d_1]$, though the sides of $W$ may vary.

Summing up the points of $\Zg_V$ are the following.

1) For each positive real $s$ the point $(s,s)$ which will be denoted by $s$.

2) For each nonnegative $q\in \Q$ the points $(q_+, q_+)$ which will be denoted by $q_+$.

We will call points from 1) and 2) \emph{ordinary}.

Furthermore for each $q\in \Q^+$ there is a pair of (topologically indistinguishable) points

3) $(q_+,q)$, denoted by $q_r$ ($r$ for right) and $(q,q_+)$ denoted by $q_u$ ($u$ for upper) - we call them
\emph{double points}.

Finally there are the following \emph{infinite points}.

5) $(\fty, \fty)$ which we denote by $\fty$.

6) The pair of topologically indistinguishable points $(\fty, 1)$, $(\fty, 1_+)$, denoted by $\fty_r$;
and a similar pair $(1, \fty)$, $(1_+, \fty)$, denoted by $\fty_u$.

The points of $\Zg_V$ are ordered in a natural way, extending the ordering of reals and setting  $q< q_r, q_u< q_+$
for every rational $q$. Furthermore we will put the infinite points $\fty_u, \fty_r$ after all finite points, but
before $\fty$. We will use black squares to denote the pairs of topologically indistinguishable points.

$$
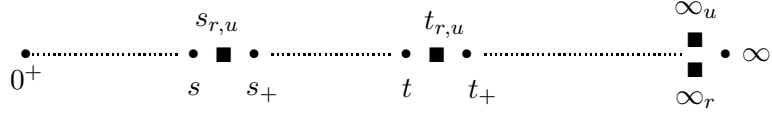
\begin{xy}
(0,0)*+={_{\bullet}}="0";(22,0)*+{_{\bullet}}="s";(26,0)*+{_{\blacksquare}};(30,0)*+{_{\bullet}}="s+";%
{\ar@{.} "0";"s"};
(0,-3)*+={0^+};(22,-4.5)*+={s};(25,4)*+={s_{r,u}};(31,-5)*+={s_+};
(50,0)*+{_{\bullet}}="t";(54,0)*+{_{\blacksquare}};(58,0)*+{_{\bullet}}="t+";%
(50,-4.5)*+={t};(55,4)*+={t_{r,u}};(60,-5)*+={t_+};
{\ar@{.}"s+";"t"};
(86,0)*+={}="fty";(92,0)*+{_{\bullet}};(88,2)*+{_{\blacksquare}};(88,-2)*+{_{\blacksquare}};%
(88,6)*+={\fty_u};(88,-6)*+={\fty_r};(96,0)*+={\fty};
{\ar@{.} "t+";"fty"};
\end{xy}
$$

\vspace{3mm}

The topology on this space resembles Sorgenfrey's topology on the real line (see \cite[Exam.~4.6]{Will}) and
is given by the following open basis.

For each pair of nonnegative rationals $s< t$ there is a half open interval $(s,t]$ whose content is the closed
interval $[s^+,t]$ on the diagram. For instance, it includes $s^+$ but excludes $s$ and $s_{r,u}$.

Furthermore there are infinite intervals $(s,\fty_r)$ including all finite points to the right of $s_+$
(inclusive) and $\fty_r$, but neither $\fty_u$ nor $\fty$; and similarly for intervals $(s,\fty_u)$.

Finally there is an interval $(s,\fty)$ which includes all finite points to the right of $s_+$ and all infinite
points.

Thus one could consider intervals $(s,t]$ as above as basic open sets in the Ziegler spectrum of $V$.
The following trivial remark will be used later. Namely, the inclusion $(s,t]\seq \cup_{i=1}^n (s_i,t_i]$
of basic open sets holds true if and only if each of the following holds.

1) There exists $i$ such that $s_i\leq s< t_i$, i.e.\ the left end $s_+$ is covered.

2) There exists $j$ such that $s_j< t\leq t_j$, i.e.\ the right end $t$ is covered.

3) For each $u$ with $s< u< t$ there is a $k$ such that $s_k< u< t_k$, i.e.\ each internal point is covered.

An obvious adjustment is required for infinite intervals.

Thus in order to decide the above inclusion it suffices to resolve strict and non-strict inequalities between
$s, t$ and the $s_i, t_i$.

\section{B\'ezout domains}\label{S-Bezout}

We recall some facts from the model theory of modules over commutative B\'ezout domains. Here a domain $B$ is said
to be \emph{B\'ezout}, if every 2-generated (and hence finitely generated) ideal of $B$ is principal. Thus $B$ is
B\'ezout if and only if for every $0\neq a, b\in B$ there are $u, v, g, h\in B$ such that the following
\emph{B\'ezout equations} hold true:

$$
au+ bv= c \qquad \text{and}\qquad cg=a, \ ch= b\,. \leqno(1)
$$

Such an element $c$ is called a \emph{greatest common denominator} of $a$ and $b$, written $\gcd(a,b)$, and is unique
up to a multiplicative unit.

It is well known that the intersection of two principal ideals of a B\'ezout domain is a principal ideal. If
$aB\cap bB= dB$ then $d$ is said to be a \emph{least common multiple} if $a$ and $b$, written $\lcm(a,b)$,
which is unique up to a multiplicative unit.

Thus for $0\neq a, b\in B$, under a suitable choice of units, we obtain the equality $ab= \gcd(a,b)\cdot \lcm(a,b)$.
In particular a single generator $d$ for the ideal $(a:b)= \{r\in B\mid br\in aB\}$ can be calculated by the formula
$d= \lcm(a,b)/b$.

If $P$ is a prime ideal of $B$ then the localization $V= B_P$ is a valuation domain.

If $B$ is a B\'ezout domain, then $\mL_B$ will denote the first order language of $B$-modules. If $a\in B$ then
$a\mid x$ is the \emph{divisibility formula} whose value on a module $M$ equals the image of $a$, $Ma$. Similarly
$xb=0$ for $b\in B$ denotes the \emph{annihilator} formula whose value $(xb=0)(M)$ is the kernel of $b$.

We need the following fact.

\begin{fact}(see \cite[Lemma~2.3]{P-T15})\label{B-eli}
Each positive primitive formula $\phi(x)$ in $\mL_B$ is equivalent to a finite sum of formulae $a\mid x\wg xb=0$,
$a, b\in B$, and to a finite conjunction of formulae $c\mid x+ xd=0$, $c, d\in B$.
\end{fact}

The following is a standard consequence of Fact~\ref{B-eli}.

\begin{cor}\label{zig-basis}
Let $B$ be a B\'ezout domain. An open basis for $\Zg_B$ is given by pairs $(\phi/\psi)$, where
$\phi\doteq a\mid x\wg xb=0$, $0\neq a, b\in B$, and $\psi\doteq c\mid x+ xd=0$, $0\neq c, d\in B$.
\end{cor}

We will recall from \cite{P-T15} how to decide when the above basic open set is non-empty. Firstly we may replace
$\phi$ by $\phi+ \psi$ without changing the open set. On the level of elements that means taking $\gcd$'s
and $\lcm$'s, hence we can assume that $c= ga$ and $b= dh$ for some $g, h\in B$. Now this open set is empty if
and only if $g$ and $h$ are coprime, $\gcd(g,h)=1$.

Now we look at this pair as follows. If $P$ is a nonzero prime ideal of $B$ such that $gh\notin P$ then the
open set $(\phi/\psi)_P$ defined by this pair over the valuation domain $V= B_P$ is trivial. Otherwise
(see Section~\ref{S-val}) we may consider this pair as a `rectangle' $(d,b]\times (a,c]$ (or rather
$(v(d),v(b)]\times (v(a),v(c)]$) on the plane $\wh\Ga^+\times \wh\Ga^+$, where $\wh\Ga^+$ is the positive cone of
$V$ completed by cuts. By `projecting onto the diagonal' (see Section~\ref{S-val}) we represent this open set
by a half-open interval $(v(ad), v(bc)]$.

Following \cite{PPT} and \cite[Ch.~17]{Preb1} we will consider the following setup for analyzing decidability
of the theory of modules.

A countable B\'ezout domain $B$ is said to be \emph{effectively given}, if its elements can be listed as
$a_0= 0$, $a_1=1$, $a_2$, $\dots$ (possibly with repetitions) such that the following algorithms can be executed effectively (when $m$ and $n$ range over nonnegative integers).

1) Deciding whether $a_m= a_n$ or not.

2) Producing $a_m+ a_n$ and $a_m\cdot a_n$ (or rather indices of these elements in the list).

3) Establishing whether $a_m$ divides $a_n$.

In fact this notion is equivalent to a standard definition of a computable algebra (see \cite[Ch.~6]{Ersh} or
\cite{Rab}).

Despite the countability assumption seems to be too restrictive in some cases, it means just the countability
of the first order language $\mL_B$. Furthermore if $B$ is written as a list, then questions 1)--3) are interpretable
in the first order theory of $B$-modules, hence must be answered effectively.

Under these assumptions the standard list of axioms for the theory of $B$-modules is recursive, hence this theory is
recursively enumerated. Furthermore the following procedures can be carried out effectively.

4) Given $a_m$ producing $-a_m$.

5) Given $a_m, a_n$ such that $a_n$ is not zero and $a_m$ divides $a_n$ producing the quotient $a_n/a_m$.

6) Given $a_m$ deciding whether $a_m$ is a unit in $B$ and, if so, computing its inverse.

7) Given nonzero $a_m, a_n$, producing a representative of their greatest common divisor and a representative
of their least common multiple.

Recall that the \emph{prime radical relation} $a\in \rad(b)$ on a commutative ring $R$ is equivalent to $a^n\in bR$
for some $n\in \N$, so in general it is not first order in the language of rings.

\begin{lemma}\label{rad-rel}
Suppose that $B$ is a B\'ezout domain such that each nonzero prime ideal is maximal. Then the principal radical
relation is first order in the language of rings.

Furthermore if $B$ is effectively given, then $a\in \rad(b)$ can be decided effectively.
\end{lemma}
\begin{proof}
We prove that this relation is equivalent to the first order statement $\fa\, x\, \gcd(a,x)=1\to \gcd(b,x)=1$ and check only the
right to left implication.

Otherwise $a\notin \rad(b)$, hence $a\notin P$ for some prime ideal containing $b$. By the assumption $P$ is maximal,
hence $ar+ x= 1$ for some $x\in P$ and $r\in B$. It follows that $\gcd (a,x)= 1$ but $b$ and $x$ are not coprime, a
contradiction.

To prove effectiveness make a list of powers $a, a^2, \dots$ and check at each step, using 3), whether $a^n\in bB$
or not. If this is the case stop with the answer $a\in \rad(b)$. Otherwise continue.

On the other hand for each $x$ that occurs so far in the list check, using 7), whether $\gcd (a,x)=1$.
If this is not the case then continue, otherwise check whether $\gcd (b,x)= 1$. If not then stop with the
$a\notin \rad(b)$ answer, otherwise continue.

Clearly these two parallel processes will provide the answer after finitely many runs.
\end{proof}

Now we are in a position to prove the main theorem.

\begin{theorem}\label{main}
Suppose that $B$ is an effectively given B\'ezout domain such that each nonzero prime ideal $P$ is maximal, the
residue field $B/P$ is infinite and the maximal ideal of the localization $B_P$ is not finitely generated. Then
the theory $T$ of all $B$-modules is decidable.
\end{theorem}
\begin{proof}
By Baur--Monk theorem (which eliminates quantifiers down to $\fa\, \ex$-level, see \cite[Sect.~2.4]{Preb2})
each sentence in $T$ is equivalent to a Boolean combination of the so-called invariant sentences.
Clearly (just knowing that this equivalence exists) this elimination procedure can be executed effectively.

Furthermore both equivalent forms from Lemma \ref{B-eli} can be found effectively (just knowing that they exist).
Because each residue field of $B$ is infinite, standard arguments (see \cite[Sect.~5]{PPT}) show that it suffices
to decide effectively the following inclusion of basic open sets:

$$
(\phi/\psi)\seq \bigcup_{i=1}^n (\phi_i/\psi_i)\,, \leqno (2)
$$

\noindent where $\phi\doteq a\mid x\wg xb=0$, $\psi\doteq c\mid x+ xd=0$, $\phi_i\doteq a_i\mid x\wg xb_i=0$,
$\psi_i\doteq c_i\mid x\wg xd_i=0$ and all elements are from $B$. Furthermore (see \cite[before Prop.~3.2]{P-T15})
we may assume that $c= ag$, $b= dh$ and $c= a_i g_i$ and $b_i= dh_i$ for some elements in $B$.

Now this inclusion clearly localizes: it holds true over $B$ if and only if it is true over any localization $B_P$,
where $P$ is a prime nonzero ideal. By the assumption the value group of each localization is countable,
dense and archimedean, hence isomorphic to $(\Q,+,\leq)$.

According to Section \ref{S-val} we will consider (2) as the following inclusion of half-open intervals:

$$
(v(ad), v(bc)]\seq \bigcup_{i=1}^n (v(a_id_i), v(b_i c_i)]\,. \leqno (3)
$$

We will decide on existence of a prime ideal $P$ such that this inclusion fails. For this we use the explanation
after Corollary \ref{zig-basis}.

First of all the interval $(v(ad), v(bc)]$ should be nontrivial over $B_P$. That means just that $g, h\in P$.
Thus we have gotten a first condition in our projected $P$.

Furthermore (see Section \ref{S-val}) the remaining conditions can be written in the form $v(s)< v(t)$ or
$v(s)\leq v(t)$ where $s, t$ take values among $ad, bc, a_id_i, b_ic_i$. Let $u$ generate the ideal $(t:s)$ and
$w$ generate the ideal $(s:t)$. Then the former equality $v(s)< v(t)$ is equivalent to $u\in P$, and the latter
is equivalent to $w\notin P$.

Summing up we see that the inclusion (3) does not hold if and only if there exists a prime nonzero ideal $P$
which includes a prescribed finite set of elements of $B$, and excludes a prescribed finite set of elements.
Using $\gcd$'s and $\lcm$'s we may assume that both sets consist of one element.

Thus we are left with the following question. Given nonzero $a, b\in B$ decide on the existence of a prime ideal
$P$ such that $a\notin P$ and $b\in P$. But this is the same as $a\notin \rad(b)$, hence Lemma~\ref{rad-rel} applies.
\end{proof}

Struggling more one could likely relax the hypotheses of Theorem \ref{main} to include the case when some residue
fields $B/P$ are finite. However because of the lack of nice examples, we will not pursue this goal.

Here is a modeling example when the situation described in Theorem \ref{main} occurs. Let $A$ be a domain
with quotient field $Q$ whose algebraic closure is denoted by $\wt Q$. Then the \emph{algebraic closure}
$\wt A$ of $A$ is its integral closure in $\wt Q$.

\begin{prop}\label{effect}
Suppose that $A$ is a countable domain such that each nonzero prime ideal of $A$ is maximal and its algebraic
closure $B$ is an effectively given B\'ezout domain. Then the theory of all $B$-modules is decidable.
\end{prop}
\begin{proof}
We will check that the hypotheses of Theorem \ref{main} are satisfied.

Firstly it follows from standard properties of integral closures (see \cite[Cor.~5.8, 5.9]{A-M}) that each
nonzero prime ideal of $B$ is maximal.

Secondly we check that the maximal ideal $P_P$ of $B_P$ is not principal. Otherwise let the image of $r\in P$
generate this ideal. Since $\sq r\in B$ we conclude that $\sq r= rt$ for some $t\in R_P$, hence $1= \sq r\cdot t$,
a contradiction.

It remains to show that each residue field $B/P$ is infinite. If the characteristic of $A$ is $p$ then $B$
includes a copy of any finite field $\F_q$, $q= p^m$, which survives after the localization.

Otherwise $A$ contains a copy of $\Z$, therefore $B$ contains the ring of algebraic integers $\Alg= \wt\Z$.
It is well known that each residue field of this ring is infinite. For instance one can argue as follows.

Let $P$ be a maximal ideal of $\Alg$, hence $P$ contains a rational prime $p$. Let $m= p^k-1$ and let $\th_m$
denote the $m$th primitive root of unity. Then $\Z[\th_m]$ is the ring of algebraic integers in the field
$\Q[\th_m]$ and (see \cite[p.~52, Exer.~5]{Jan}) $p$ completely ramifies in $\Z[\th_m]$.

Furthermore if $p\Z[\th_m]= M_1\cdot\ldots\cdot M_l$ for maximal ideals $M_i$, then each field $\Z[\th_m]/M_i$
contains $p^k$ elements, hence the same is true for $B/P$. Since $k$ is arbitrary, we obtain the desired.
\end{proof}

If $A$ is completely integrally closed in its quotient field, then some necessary conditions for $\wt A$ to be
B\'ezout are pointed out in \cite[Prop.~5.2]{D-M}. For instance $A$ should be Pr\"ufer with torsion Picard group.
The latter means that for each finitely generated (and hence projective) ideal $I$ of $A$ some power $I^n$ ($n$ a
positive integer) is a principal ideal. For instance (see \cite[Exam.~5.3]{D-M}) if a field $F$ is not algebraic
over a finite field and $A= F[t]$, then the ring $\wt A$ is not B\'ezout.

On the other hand if $B= \wt A$ is a union of Pr\"ufer domains $B_i$ with torsion Picard groups then it is B\'ezout.
Namely if nonzero $a, b$ are in the same $B_i$, then choose $n$ such that the ideal $(aB_i+bB_i)^n$ of $B_i$ is
principal generated by $c$. Then $d= \sq[n] c$ generates $aB+ bB$. This clearly applies to $\Alg$ and also to the
algebraic closure of the polynomial ring $\F_p[t]$, as well as their localizations. Furthermore (see \cite[5.5]{D-M})
it works if $A$ is the localization $\wt \Q[t]_{(t)}$ of $\wt \Q [t]$ with respect to the ideal $t\,\wt \Q [t]$.

To apply Proposition \ref{effect} (to prove decidability) we need the effectiveness of the algebraic closure of
$A$, which is a subtle question. If $A$ is effectively given then the same obviously holds for its quotient ring
$Q$. As follows from \cite[Thm.~7]{Rab} the algebraic closure $\wt Q$ is also effectively given. Actually the
construction of this paper creates a new one-to-one numeration for $\wt Q$ such that even recovering $Q$ inside
this field is not straightforward. Namely by \cite[Lemma~6]{Rab} this is the case if and only if $Q$ possesses a
splitting algorithm (for polynomials in one variable).

On the basis of this we obtain.

\begin{prop}\label{rub}
Suppose that $A$ is a countable effectively given domain with the quotient field $Q$ which possesses a splitting
algorithm for polynomials in one variable. If the algebraic closure $B$ of $A$ is B\'ezout, then the  theory of
all $B$-modules is decidable.
\end{prop}
\begin{proof}
By Proposition \ref{effect} it suffices to create an effective list for elements of $B$.

Start with an effective list of elements of $A$. Clearly this extends to an effective list of elements of $Q$.
As in \cite{Rab} create an effective list of the algebraic closure of $Q$ such that $Q$, and hence $A$ can be effectively
recovered within this list.

Make a list of monic polynomials $f(x)$ in one variable with coefficients in $A$. For a parallel list of elements 
$a\in A$ check whether $f(a)= 0$. If this holds true, include $a$, otherwise drop $f$ and continue.

This way we will create a list of elements of $B$. The operations on $B$ are effective, because they are effective in
$\wt Q$.
\end{proof}

Now we are in a position to analyze examples of our original interest.

\begin{theorem}\label{a-decid}
1) The theory of all modules over the ring $\Alg$ of algebraic integers in decidable.

2) The same is true for the algebraic closure of the polynomial ring $\F_p[t]$ ($p$ a prime).
\end{theorem}
\begin{proof}
We start with $A= \Z$ in the case of algebraic integers, and with $A= \F_p[t]$ in the second case, both are effectively
given principal ideal domains whose algebraic closures are B\'ezout. To apply Proposition \ref{rub} we need a
splitting algorithm for their quotient fields. In case of integers $Q$ is the field of rationals, hence such 
algorithm is known from Kronecker.

In the second case $Q$ is the function field $\F_p(t)$, hence the required algorithm can be extracted for instance from \cite{Len}.
Namely let $f(x)$ be a polynomial whose coefficients are rational functions. By clearing denominators we may assume
that these coefficients are in $\F_p[t]$. Since $\F_p[t]$ is a factorial ring, by Gauss' lemma, this polynomial is
irreducible over $\F_p(t)$ iff it is irreducible over $\F_p[t]$, i.e. as a polynomial in two variables. However the
splitting algorithm for such polynomials over finite fields is known from Lenstra \cite{Len}. Of course factoring in $\F_p[t][x]$ is not just the same as factoring in $\F_p[t,x]$, because, for instance, $t$ is a coefficient in the former case and an irreducible factor in the latter; but a factorization in $\F_p[t,x]$ clearly determines that in $\F_p[t][x]$.
\end{proof}

Note that in both cases one could construct the algebraic closure directly, hence apply Proposition \ref{effect}.

Namely in the case of $\A$ (see \cite[Ch.~3]{Coh}) each algebraic integer can be given by a monic irreducible
polynomial $f(x)\in \Z[x]$ and an approximation good enough to separate it from other roots of $f$ (this
approximation can be found very effectively - see \cite[p.~132]{Rum} or \cite{Pet}).

To execute the addition $a+b$ of algebraic integers defined by polynomials $f(x)$ and $g(x)$ we calculate the
resultant $h(x)$ of $f(x)$ and $g(x-y)$ considered as polynomials in $y$ whose coefficients are in $\Z[x]$.
Then $a+b$ is a root of $h(x)$. To separate it from the remaining roots of this polynomial we might need a
better approximation for $a$ and $b$, but this is easily arranged.

Similarly the product $a\cdot b$ is a root of the resultant $y^m f(x/y)$ and $g(y)$, where $m$ is the degree of
$f(x)$.

Finally the ratio $c= a/b$ (which is an algebraic number) is a root of the resultant of $f(xy)$ and $g(x)$.
To decide whether this ratio is an algebraic integer one could find, calculating in $\Q(a,b)$, the minimal
polynomial of $c$ and check whether it is monic.

In the case $\F_p[t]$ we are less certain. It appears (see \cite[after 1.11]{D-M}) that the algebraic closure of
this ring equals to the union of fields $\F_{p} (t^{1/p^d})[x]/(f)$, where $f(x)$ is an irreducible polynomial;
however it is difficult to find a proof of this result. Note that nontrivial polynomials occur: for instance 
(see \cite{Ked}) the Artin--Schreier polynomial $f= x^p- x- t^{-1}$ has no roots even in the corresponding power 
series field.

\begin{ques}
Is the theory of all $\wt{\Q[t]_{(t)}}$-modules decidable?
\end{ques}

We will complete this section by the following conjecture.

\begin{conj}\label{a-spec}
The Ziegler spectrum over the ring of algebraic integer $\Alg$ is a spectral space.
\end{conj}

\section{Conclusions}

Recall that Gregory \cite{Gre15} proved that the theory of all modules over an effectively given valuation
domain $V$ (with a prescribed size of the residue field) is decidable if and only if the prime radical
relation in $V$ is computable. Bearing this in mind one would expect that the following metatheorem can be
proven. Suppose that $B$ is an effectively given countable B\'ezout domain such that the prime radical relation
is decidable and few (yet unspecified) questions on finite invariants (and hence on the maximal spectrum of
$B$) can be answered effectively. Then the first order theory of all $B$-modules is decidable.

However taking into account the very technical proofs in \cite{Gre15} one should be urged by nice examples to
embark on such a journey. In our opinion one should be guided by examples to create theories. From this point
of view dropping the countability of the ring assumption is a more urgent issue.

Given a sentence $\phi$ in the language of modules over the ring $R$, it mentions only finitely many names
$\ov r$ of elements of $R$. Could one (uniformly in $\ov r$) decide, using only finite fragments of the ring,
whether this sentence holds true in each $R$-module?

The following instance of this question seems to be of particular importance in the spirit of this paper.

\begin{ques}
Let $E$ be the ring of entire complex or real functions. Is the theory of all $E$-modules decidable?
\end{ques}

Another perspective line of research would be to consider a combined two-sorted language of rings and modules
(see \cite[Ch.~9]{J-L}). Guided by the decidability of the theory of algebraic integers in the language of
rings (see \cite{Dri}) and Theorem \ref{a-decid} we dare to ask the following.

\begin{ques}
Is the first order theory of algebraic integers in the two-sorted language decidable?
\end{ques}

\end{document}